\documentclass{article}
\pdfoutput=1
\usepackage{spconf,amsmath,graphicx}
\usepackage{amsfonts}       
\usepackage{amsmath,amsthm,amssymb}
\usepackage{color}
\usepackage{bm}

\newtheorem{theorem}{Theorem}

\newtheorem{proposition}{Proposition}
\newtheorem{fact}{Fact}
\newtheorem{definition}{Definition}
\usepackage{balance}

\title{Finite sample deviation and variance bounds for \linebreak first order autoregressive processes}

\name{Rodrigo A. Gonz\'alez and Cristian R. Rojas \thanks{This work was supported by the Swedish Research Council under contract number 2016-06079 (NewLEADS).
		E-mails: $\{grodrigo@kth.se, crro@kth.se\}$.}}
\address{Division of Decision and Control Systems\\KTH Royal Institute of Technology, Stockholm, Sweden }

\begin{document}
\ninept
\maketitle

\begin{abstract}
In this paper, we study finite-sample properties of the least squares estimator in first order autoregressive processes. By leveraging a result from decoupling theory, we derive upper bounds on the probability that the estimate deviates by at least a positive $\varepsilon$ from its true value. Our results consider both stable and unstable processes. Afterwards, we obtain problem-dependent non-asymptotic bounds on the variance of this estimator, valid for sample sizes greater than or equal to seven. Via simulations we analyze the conservatism of our bounds, and show that they reliably capture the true behavior of the quantities of interest.
\end{abstract}


\section{Introduction}
\label{sec:intro}

Identification of autoregressive (AR) processes is a well-developed research topic, which is deeply connected with the areas of probability and statistics, time series analysis, econometrics, system identification, signal processing, machine learning and control. 

In AR processes, estimation of parameters given data is commonly done by the least squares method, which is known to have good statistical performance and is related to maximum likelihood in a Gaussian noise framework. Consequently, the asymptotic properties of this method for AR processes have been studied thoroughly, starting since at least \cite{mann1943statistical}. In general regression models, a detailed consistency analysis of the least squares estimate was given in \cite{lai1983asymptotic}.

Despite the popularity of least squares, non-asymptotic performance bounds with this method are rare in the literature. Finite-time bounds are important for computing the number of samples needed for achieving a specified accuracy, deriving finite-time confidence sets, and designing robust control schemes. Most of the difficulties in obtaining these bounds lie on the limitations of classical statistical techniques, which are better suited for asymptotic analyses.

Recently, results in statistical learning theory, self-normalizing processes \cite{pena2008self} and high dimensional probability \cite{wainwright2019high}, have motivated control, signal processing, and machine learning communities to explore finite-time properties of linear system estimates by least squares.
Among topics of interest, we can find sample complexity bounds \cite{jedra2019sample}, $1-\delta$ probability bounds on parameter errors \cite{sarkar2019near}, confidence bounds \cite[Chap. 20]{lattimore2018bandit}, Cram\'{e}r-Rao lower bounds \cite{friedlander1989exact}, and bounds on quadratic criterion cost deviation \cite{weyer1999finite,campi2002finite}. In these works general autoregressive structures often appear, but the powerful statistical machinery used usually leads to loss of insight in scalar cases, where specialized theory could serve better. A problem-dependent sharp analysis of AR(1) processes is therefore needed, as it can be considered as a natural stepping stone towards studying finite-time properties of more general model structures such as ARX and ARMAX models, and can also be useful for analyzing finite-time properties of two-stage ARMA estimation algorithms \cite{stoica2005spectral}.

In summary, the main results of this paper are:

\begin{itemize}
\item
For both stable and unstable single-parameter scalar autoregressive processes, we derive upper bounds on the deviation probability of the least-squares estimate;

\item We provide interpretable upper bounds on the variance of this estimated parameter for both aforementioned scenarios; 

\item We show that our variance bound for the stable case matches the well-known asymptotic linear decay; and for unstable systems, the variance bound decays exponentially on the number of data points, with rate of decay proportional to the magnitude of the unstable pole;

\item
We illustrate the applicability of these bounds via Monte Carlo simulations, and show that they manifest the correct qualitative behavior of deviation probability and variance.
\end{itemize}

The rest of this paper is organized as follows. We put our work into context in Section \ref{sec:prior}, and explicitly formulate the problem in Section \ref{sec:problemformulation}. Preliminaries from decoupling theory are introduced in Section \ref{sec:preliminaries}. After this, we state and prove our non-asymptotic bounds for stable and unstable processes in Sections \ref{sec:stable} and \ref{sec:unstable} respectively. Simulations and conclusions are presented in Sections \ref{sec:simulations} and \ref{sec:conclusions}.



\section{RELATION TO PRIOR WORK}
\label{sec:prior}
In signal processing, Cram\'{e}r-Rao lower bounds for finite number of measurements were studied in \cite{friedlander1989exact}. In a time-series context, \cite{bercu1997large,bercu2001large} studied large deviation rates of the least squares estimator in an AR(1) process. In adaptive control,
a bound on the estimation error of the parameter of a scalar first order system was given in \cite{rantzer2018concentration}, together with an upper bound on the variance of the estimate. Although similar to our work, these bounds do not depend on the real system parameter. This leads to a) conservatism, as usually the worst-case-scenario is taken into account, and b) lack of insight on the effect of stability margins in the accuracy of the estimate. A deviation bound for finite time that is also problem independent was given in \cite{bercu2008exponential}. 

In a broader context, one of the first non-asymptotic results in system identification was presented in \cite{campi2002finite}, where a uniform bound for the difference between empirical and theoretical identification costs was obtained. More recently, among works that have analyzed finite-time identification for stochastic processes are \cite{sarkar2019near,simchowitz2018learning,faradonbeh2018finite,tsiamis2019finite}. These papers have focused on vector autoregressive models described by matrix parameters in the form of state-space realizations, which do not reduce to interpretable and tight bounds for scalar cases. Only \cite{simchowitz2018learning} studied a scalar system separately, and derived probability $1-\delta$ error guarantees depending on the real parameter value. However, as opposed to our bounds, these results are rather difficult to interpret, and no variance bound is obtained. 

\section{Problem formulation}
\label{sec:problemformulation}
Consider the AR(1) process
\begin{equation}
\label{ar1}
y_t = a_0 y_{t-1}+e_t,
\end{equation}
where $\{e_t\}$ is a Gaussian white noise process of variance $\sigma^2$, and $a_0$ is unknown. Given $N$ data points $\{y_t\}_{t=1}^N$, the least squares estimate of $a_0$ \cite[Chap. 3]{stoica2005spectral} is
\begin{equation}
\label{ls1}
\hat{a}_N := \dfrac{\sum\limits_{t=2}^N y_t y_{t-1}}{\sum\limits_{t=1}^{N-1} y_{t}^2} = a_0 + \dfrac{\sum\limits_{t=2}^N e_t y_{t-1}}{\sum\limits_{t=1}^{N-1} y_{t}^2}.
\end{equation}
It is challenging to derive an analytical expression for the distribution of the random variable $\hat{a}_N-a_0$. Therefore, instead of exact expressions, it is of interest to obtain an upper bound for the deviation of $\hat{a}_N$ around the true value $a_0$ for finite samples. Due to the symmetry of the distribution, our interest lies on upper bounding
\begin{equation}
\mathbb{P}(\hat{a}_N-a_0>\varepsilon) \notag
\end{equation}
for any positive and fixed $\varepsilon$. Provided this probability bound, the next natural step is to compute an upper bound on the variance of least squares for finite samples in this scenario, which is defined as
\begin{equation}
\textnormal{var}\{\hat{a}_N\} = \mathbb{E}[(\hat{a}_N-a_0)^2]. \notag
\end{equation} 
The goal of this paper is to compute both bounds under two different regimes: stable $(|a_0|<1)$ and strictly unstable ($|a_0|>1$). For this purpose, we make use of an exponential inequality for martingales, which we detail next.

\section{Preliminaries: Theory of decoupling}
\label{sec:preliminaries}
In this section we give a brief summary of decoupling theory, mostly based on \cite{victor1999general}. Decoupling theory provides an approach to handling complex problems were dependent variables are involved. By breaking the dependence structure, it is possible to tackle inequalities with greater ease. Before presenting the result we have used, we state two important definitions.

\begin{definition}[$\{\mathcal{F}_i\}$-tangent sequence]
Let $\{d_i\}_{i\in \mathbb{N}}$, $\{e_i\}_{i\in \mathbb{N}}$, be two sequences of random variables adapted to an increasing sequence of $\sigma$-fields $\mathcal{F}_i$, and assume that $\mathcal{F}_0$ is the trivial $\sigma$-field. Then $\{e_i\}$ is said to be $\{\mathcal{F}_i\}$-tangent to $\{d_i\}$ if for all $i$,
\begin{equation}
p(d_i|\mathcal{F}_{i-1}) = p(e_i|\mathcal{F}_{i-1}), \notag
\end{equation} 
where $p(d_i|\mathcal{F}_{i-1})$ denotes the conditional probability law of $d_i$ given $\mathcal{F}_{i-1}$.
\end{definition}  

\begin{definition}[Condition CI]
A sequence $\{e_i\}$ of random variables adapted to an increasing sequence of $\sigma$-fields $\mathcal{F}_i$ contained in $\mathcal{F}$ is said to satisfy \textnormal{condition CI} if there exists a $\sigma$-algebra $\mathcal{G}$ contained in $\mathcal{F}$ such that $\{e_i\}$ is a sequence of conditionally independent random variables given $\mathcal{G}$ and
\begin{equation}
p(e_i|\mathcal{F}_{i-1})= p(e_i|\mathcal{G}) \notag
\end{equation} 
for all $i$. The sequence $\{e_i\}$ is then said to be \textnormal{decoupled}.
\end{definition}
It is known \cite{de2012decoupling} that there always exists a sequence $\{d_i\}_{i=1}^n$ of conditionally independent random variables given a master $\sigma$-algebra $\mathcal{G}$ (where usually $\mathcal{G}=\sigma(\{e_i\})$) that has the same individual marginals as the original sequence. In particular, $d_i$ has the same conditional distribution as $e_i$, for all $i=1,\dots, n$, since it can be seen as drawn from independent repetitions of the same random mechanism. This construction can be made sequentially \cite{victor1999general}. 

We now are ready to state the key Proposition used in this work.

\begin{proposition}
	\label{proposition1}\textnormal{\cite[Corollary 3.1]{victor1999general}}
Let $\{d_i\},\{e_i\}$ be $\{\mathcal{F}_i\}$-tangent. Assume that $\{e_i\}$ is decoupled (i.e., condition CI is satisfied). Let $g\geq 0$ be any random variable measurable with respect to $\sigma(\{d_i\}_{i=1}^\infty)$. Then for all finite $t$,
\begin{equation}
\mathbb{E}\left[ g \exp\left\{ t\sum_{i=1}^n d_i \right\}  \right] \leq \sqrt{\mathbb{E}\left[ g^2 \exp\left\{ 2t\sum_{i=1}^n e_i \right\}  \right]}. \notag
\end{equation} 
\end{proposition}

\section{Bounds for stable AR(1) processes}
\label{sec:stable}
In this section, we derive the first key result of this paper, Theorem \ref{theorem1}, which provides deviation probability and variance bounds for a stable first order autoregressive process and stationary data $\{y_t\}_{t=1}^N$. We do this by exploiting Proposition \ref{proposition1}.

\begin{theorem} \label{theorem1}
	Consider the AR(1) process described by \eqref{ar1}, and the least-squares estimate \eqref{ls1}. Assume that $\{y_t\}$ is stationary and $|a_0|<1$. Then, for all $\varepsilon>0$ and sample size $N\geq 2$,
	\begin{align}
	\mathbb{P}(&\hat{a}_N-a_0>\varepsilon)\leq \left(\frac{1-a_0^2}{1-a_0^2+\varepsilon^2}\right)^{\frac{1}{4}} \times \notag \\
	\label{pbstable1}
	&\left(\frac{1+a_0^2 + \varepsilon^2}{2}+\sqrt{\left(\frac{1+a_0^2 + \varepsilon^2}{2}\right)^2-a_0^2} \right)^{\frac{-(N-2)}{4}}.
	\end{align}
	Furthermore, for all $N\geq 7$,
	\begin{equation}
	\label{varstable1}
	\textnormal{var}\{\hat{a}_N\}\leq \frac{8}{N-6}-\frac{8a_0^2}{N+2}.
	\end{equation}
\end{theorem}

\begin{proof}[Proof sketch]
 Given $\varepsilon>0$, we have that
	\begin{align}
	\mathbb{P}(\hat{a}_N-a_0>\varepsilon) &= \mathbb{P}\left(\sum_{t=2}^N e_t y_{t-1}-\varepsilon \sum_{t=1}^{N-1} y_{t}^2 > 0\right) \notag \\
	&\leq \inf_{\lambda>0} \mathbb{E}\left[ \exp\left\{ \lambda \sum_{t=2}^N (e_t y_{t-1}-\varepsilon y_{t-1}^2) \right\} \right], \notag
	\end{align}
	where we have used Chernoff's inequality \cite[Chap. 2]{wainwright2019high}. Next, we need a bound on the right hand side. This can be derived by constructing a decoupled sequence and applying Proposition \ref{proposition1}, as follows: let $d_t \in \mathcal{F}_{t}:=\sigma\{e_1,\dots,e_{t}\}$ with the same $\mathcal{F}_{t-1}$-conditional distribution as $e_t y_{t-1}$. Moreover, let the sequence $\{d_t\}$ be conditionally independent given the master $\sigma$-field $\mathcal{G}:= \sigma(\{e_t\})$. Given this decoupled sequence, from Proposition \ref{proposition1} we have that
	
\begin{align}
	\inf_{\lambda>0}\mathbb{E}\left[ e^{ \lambda \sum_{t=2}^N (e_t y_{t-1}-\varepsilon y_{t-1}^2)} \right]&\leq \inf_{\tilde{\lambda}>0}\sqrt{\mathbb{E}\left[ e^{ \tilde{\lambda} \sum_{t=2}^N (d_t-\varepsilon y_{t-1}^2)} \right]} \notag \\
	&=\inf_{\tilde{\lambda}>0} \sqrt{\mathbb{E}\left[ e^{ \left(\frac{\tilde{\lambda}^2 \sigma^2}{2}-\varepsilon \tilde{\lambda}\right) \sum_{t=1}^{N-1} y_{t}^2 } \right]}, \notag
	\end{align}
	where we have used the fact that $\{y_t\}_{t=1}^{N-1}$ is $\mathcal{G}$-measurable. Now, $\tilde{\lambda}^2 \sigma^2/2-\varepsilon \tilde{\lambda}$ is minimum at $\tilde{\lambda}^{\textnormal{opt}} = \varepsilon/\sigma^2>0$, giving $(\tilde{\lambda}^{\textnormal{opt}})^2 \sigma^2/2-\varepsilon \tilde{\lambda}^{\textnormal{opt}} = -\varepsilon^2/(2\sigma^2)$. Thus, we have found the following bound, valid for stable and unstable systems:
	\begin{equation}
	\label{unstableandstablebound}
	\mathbb{P}(\hat{a}_N-a_0>\varepsilon) \leq \sqrt{\mathbb{E}\left[ \exp\left\{ -\frac{\varepsilon^2}{2\sigma^2} \sum_{t=1}^{N-1} y_{t}^2 \right\} \right]}.
	\end{equation}
	To further compute this bound for the stable case, notice that $\begin{bmatrix}
	y_1 & \cdots & y_{N-1}
	\end{bmatrix}^\top \sim \mathcal{N}(0,\mathbf{R}_{N-1})$, where $\mathbf{R}_{N-1}$ is a Toeplitz covariance matrix. By computing the moment generating function in \eqref{unstableandstablebound} using \cite[Theorem 3.2a.1]{mathai1992quadratic}, we find that
	\begin{equation}
	\label{generalizedchi2}
	\hspace{-0.04cm}\hspace{-0.2cm}\sqrt{\mathbb{E}\left[ \exp\left\{ -\frac{\varepsilon^2}{2\sigma^2} \sum_{t=1}^{N-1} y_{t}^2 \right\} \right]} = \textnormal{det}^{-\frac{1}{4}}\left( \mathbf{I}_{N-1}+\frac{\varepsilon^2}{\sigma^2}\mathbf{R}_{N-1}\right). \hspace{-0.1cm}
	\end{equation}
	From now on, we denote the symmetric Toeplitz matrix $\mathbf{I}_{N-1}+(\varepsilon^2/\sigma^2)\mathbf{R}_{N-1}$ as $\mathbf{T}_{N-1}$. The final step consists in finding a lower bound for det($\mathbf{T}_{N-1}$), for which we rely on a result from estimation theory. As minimum values of one-step prediction errors of increasing predictor complexity \cite{gray2006toeplitz}, it is known that	
	\begin{equation}
	\frac{\det(\mathbf{T}_{N+1})}{\det(\mathbf{T}_N)} \leq \frac{\det(\mathbf{T}_{N})}{\det(\mathbf{T}_{N-1})}, \quad N>1, \notag
	\end{equation}
	and also
	\begin{equation}
	\lim_{N \to \infty} \frac{\det(\mathbf{T}_{N+1})}{\det(\mathbf{T}_N)} = \exp\left\{ \frac{1}{2\pi} \int_{-\pi}^\pi \log\left(1+\frac{\varepsilon^2}{\sigma^2}\Phi_{y}(e^{j\omega})\right)d\omega \right\}, \notag
	\end{equation}
	where in this case, $\Phi_{y}(e^{j\omega})=\sigma^2/|e^{j\omega}-a_0|^2$. Hence, we can lower bound $\det(\mathbf{T}_{N-1})$ as follows:
	\begin{align}
	&\det(\mathbf{T}_{N-1})=\frac{\det(\mathbf{T}_{N-1})}{\det(\mathbf{T}_{N-2})} \frac{\det(\mathbf{T}_{N-2})}{\det(\mathbf{T}_{N-3})} \dots \frac{\det(\mathbf{T}_2)}{T_1} T_1 \notag \\
	\label{szegointegral}
	&\hspace{0.15cm}\geq \left[\exp\left\{\frac{1}{2\pi} \int_{-\pi}^\pi \log\left(1+\frac{\varepsilon^2}{|e^{j\omega}-a_0|^2}\right) d\omega \right\} \right]^{N-2} T_1. 
	\end{align}
Here, $T_1$ can be computed as $(1-a_0^2+\varepsilon^2)/(1-a_0^2)$, while the integral in \eqref{szegointegral} can be evaluated using Szeg\"o's formula \cite{gray2006toeplitz}, according to which the integral is equal to the minimum variance of the prediction error for a process with spectrum $\Phi_{\tilde{y}}(e^{j\omega}):=1+\varepsilon^2|e^{j\omega}-a_0|^{-2}$. By this result and \eqref{generalizedchi2}, we obtain \eqref{pbstable1}.
	
	Finally, the upper bound of the estimator's variance for finite samples can be computed through the following inequality:
	\begin{align}
	\textnormal{var}&\{\hat{a}_N\}= \int_0^\infty \mathbb{P}\{(\hat{a}_N-a_0)^2>x\} dx \notag \\
	&\leq 2 \int_0^\infty \left( \frac{1+a_0^2 + x}{2}+\frac{1}{2}\sqrt{(1+a_0^2 + x)^2-4a_0^2} \right)^{\frac{-(N-2)}{4}} \hspace{-0.2cm}dx \notag \\
	&= \frac{8}{N-6}-\frac{8a_0^2}{N+2}. \tag*{\qedhere}
	\end{align}
\end{proof}
Note that our probability bound does not depend on the noise variance, which is intuitive since the random variable $\{\hat{a}_N-a_0\}$ is unaffected by scale changes. Regarding variance, our bound captures the correct behavior of the least squares estimate for this scenario, since stable systems that are closer to the stability limit are in fact easier to learn due to a greater signal-to-noise ratio, as argued in \cite{simchowitz2018learning}. By comparing \eqref{varstable1} to the well-known asymptotic Cram\'{e}r-Rao variance lower bound for this case, which says that for large $N$,
\begin{equation}
\textnormal{var}\{\hat{a}_N\} \approx \frac{1-a_0^2}{N-1}, \notag
\end{equation}
we see that the variance bound we have obtained has the same structure up to a positive constant factor, indicating that it is rate-optimal. Finally, $N\geq 7$ is needed since the integral of the probability bound solved for the variance does not converge for $N\leq 6$. 

\section{Bounds for unstable AR(1) processes}
\label{sec:unstable}
In the previous section, we provided a rigorous finite-sample analysis of a stable AR(1) process. We will now consider the unstable case, where $|a_0|>1$. Contrary to the previous scenario, stationarity cannot be assumed. Despite this, a similar analysis still holds, and we are able to show that the variance bound decays exponentially in the number of samples, in contrast with the stable case. We state the result in Theorem \ref{theorem2}.

\begin{theorem}
\label{theorem2}
Consider the AR(1) process described by \eqref{ar1} where $|a_0|>1$ and $y_0=0$, and consider the least-squares estimate \eqref{ls1}. Then, for all $\varepsilon>0$ and sample size $N\geq 2$,
	\begin{equation}
	\label{pbunstable1}
	\mathbb{P}(\hat{a}_N-a_0>\varepsilon)\leq\ \left(\frac{\lambda_2-\lambda_1}{(\lambda_2-1)\lambda_1^{N}-(\lambda_1-1)\lambda_2^{N}}\right)^{\frac{1}{4}}, 
	\end{equation}
	where 
	\begin{equation}
	\label{roots}
	\lambda_{1,2} = \frac{1+a_0^2+\varepsilon^2}{2}\mp\frac{\sqrt{(1+a_0^2+\varepsilon^2)^2-4a_0^2}}{2}, \hspace{0.2cm} \lambda_2>\lambda_1.  
	\end{equation}
	Furthermore, for all $N\geq 7$,
	\begin{align}
	\textnormal{var}&\{\hat{a}_N\}\leq |a_0|^{-\frac{2N}{5}}\times \notag \\
	\label{varunstable1}
	&\left[ \frac{2a_0^4(N+5)}{N} + 8\sqrt[4]{\frac{N}{N+5}}\left( \frac{a_0^2}{N-6} -\frac{1}{N+2}\right) \right].
	\end{align}
\end{theorem}

\begin{proof}[Proof sketch]
By following the same lines as the stable case proof, we reach \eqref{unstableandstablebound}. Since $y_0=0$, we find that $\begin{bmatrix}
	y_1 & \cdots & y_{N-1}
	\end{bmatrix}^\top \sim \mathcal{N}(0,\bar{\mathbf{R}}_{N-1})$, where $\bar{\mathbf{R}}_{N-1}$ is a covariance matrix with entries
	\begin{equation}
	\bar{r}_{ij}=a_0^{|i-j|} \sigma^2 \frac{a_0^{2\min\{i,j\}}-1}{a_0^2-1}, \quad i,j=1,\dots,N-1. \notag
	\end{equation}
This covariance matrix can be shown to satisfy $\sigma^2\bar{\mathbf{R}}_{N-1}^{-1} = \bar{\mathbf{T}}_{N-1}-a_0^2 \bm{\eta}_{N-1}\bm{\eta}_{N-1}^\top$, where $\bm{\eta}_{N-1}\in \mathbb{R}^{N-1}$ denotes the $N-1$-th column of the identity matrix of size $N-1$, and $\bar{\mathbf{T}}_{N-1}$ is a non-singular tridiagonal Toeplitz matrix with tridiagonal formed by $\{-a_0,a_0^2+1,-a_0\}$. Thus, by similar computations as in \eqref{generalizedchi2}, we reach
	\begin{align}
	\mathbb{P}(\hat{a}_N-a_0>\varepsilon) &\leq \textnormal{det}^{-\frac{1}{4}} \left(\mathbf{I}_{N-1}+\frac{\varepsilon^2}{\sigma^2}\bar{\mathbf{R}}_{N-1} \right) \notag \\
	&=\left(\frac{\det (\bar{\mathbf{T}}_{N-1}-a_0^2 \bm{\eta}_{N-1}\bm{\eta}_{N-1}^\top)}{\det(\bar{\mathbf{T}}_{N-1}+\varepsilon^2\mathbf{I}_{N-1}-a_0^2 \bm{\eta}_{N-1}\bm{\eta}_{N-1}^\top)}\right)^{\frac{1}{4}}. \notag
	\end{align} 
	Using Cauchy's rank-one perturbation formula \cite{Horn2012}, and the fact that $\bm{\eta}_{N-1}^\top \bar{\mathbf{T}}_{N-1}^{-1}\bm{\eta}_{N-1} = \det(\bar{\mathbf{T}}_{N-2})/\det(\bar{\mathbf{T}}_{N-1})$, we can write the previous inequality as
	\begin{align}
	\mathbb{P}&(\hat{a}_N-a_0>\varepsilon) \leq \notag \\
	&\left(\frac{\textnormal{det} (\bar{\mathbf{T}}_{N-1})-a_0^2\det (\bar{\mathbf{T}}_{N-2})}{\det (\bar{\mathbf{T}}_{N-1}+\varepsilon^2\mathbf{I}_{N-1})-a_0^2\det (\bar{\mathbf{T}}_{N-2}+\varepsilon^2\mathbf{I}_{N-2})}\right)^{\frac{1}{4}}. \notag 
	\end{align}
	The rest of the proof consists in calculating the determinants, which can be done by exploiting properties of continuants \cite[p. 564]{muir1960treatise}. These computations lead to the bound in \eqref{pbunstable1}. 
	
	In order to obtain a bound on the variance, we first determine another probability bound which is easier to integrate. By manipulating the roots $\lambda_{1,2}$ in \eqref{roots}, we find that $\lambda_1<1<\lambda_2$, and hence, we can relax the deviation bound to
	\begin{align}
	\label{bound1}
	\mathbb{P}(\hat{a}_N-a_0>\varepsilon) \leq \min\left\{1,\frac{1}{\lambda_2^{N/4}} \left(\frac{\lambda_2-\lambda_1}{1-\lambda_1}\right)^m  \right\}, 
	\end{align}
	where $m \geq 1/4$ is a fixed value which we will set later. For the next calculations, we take a fixed positive number $x^*$, which we will also determine later, and 
	\begin{equation}
	z^*:=\frac{1+a_0^2+x^*}{2} + \frac{1}{2}\sqrt{(1+a_0^2+x^*)^2-4a_0^2}. \notag
	\end{equation}
	By substitution with this same relation, we compute
	\begin{align}
	\textnormal{var}\{\hat{a}_N\}&\leq 2 x^*+2 \int_{z^*}^\infty \sqrt[4]{\frac{z^2-a_0^2}{z-a_0^2}} z^{-\frac{N}{4}-2}(z^2-a_0^2) dz \notag \\
	\label{almostbound}
	&\leq 2 x^*+\frac{8}{\sqrt[4]{z^*-a_0^2}} \left(\frac{|a_0|^{3-\frac{N}{2}}}{N-6} -\frac{|a_0|^{1-\frac{N}{2}}}{N+2}\right),
	\end{align}
	where we have used that $z^*>a_0^2$, which is justified because
	\begin{equation}
	z^* =\frac{1}{2}\left( 1+a_0^2 + x^* + \sqrt{(1-a_0^2-x^*)^2+4x^*} \right) > a_0^2+x^*. \notag
	\end{equation}
	Next, we must find an appropriate value for $x^*$ and $z^*$. The idea behind our approach is to make $x^*$, $z^*$ (approximately) coincide with the change of regime of the bound of \eqref{bound1}. It can be shown that valid values for $x^*$ and $z^*$ are
	\begin{equation}
	x^*< \frac{|a_0|^{4-\frac{N}{2m}}(N+4m)}{N}, \quad z^*=a_0^2+\frac{|a_0|^{4-\frac{N}{2m}}(N+4m)}{N}. \notag
	\end{equation}
By replacing these quantities into \eqref{almostbound}, we can obtain the best rate of decay by choosing $m = 5/4$, leading to the bound in \eqref{varunstable1}.
\end{proof}

This bound provides the intuitive decay for unstable systems: as the sample size increases, the variance of the least squares estimate decreases exponentially, and the rate of decay is proportional to the degree of instability of the system.

\section{Simulations}
\label{sec:simulations}
To validate our theoretical results, we show how the bounds obtained in Theorems 1 and 2 perform in practice. First, we test the probability bounds given in \eqref{pbstable1} and \eqref{pbunstable1}. With $5\cdot 10^{5}$ Monte Carlo runs, we have obtained sample probabilities for the event $\{\hat{a}_N-a_0>\varepsilon\}$ for $\varepsilon$ ranging from $0.01$ to $5$, and $N$ from $2$ to $100$. Both simulated and upper bound surfaces for $a_0 = 0.5, 0.98, 1.01$ and $1.1$ can be seen in Fig. \ref{fig1}. We have also obtained the sample variance for the same number of Monte Carlo runs, and from $N=7$ to $N=1000$. Fig. \ref{fig2} shows the results for the same values of $a_0$ previously studied. 

\begin{figure}[h!]
	\centering{
		\includegraphics[trim=0.0cm 0cm 0.cm 0cm,clip, width=0.48\textwidth]{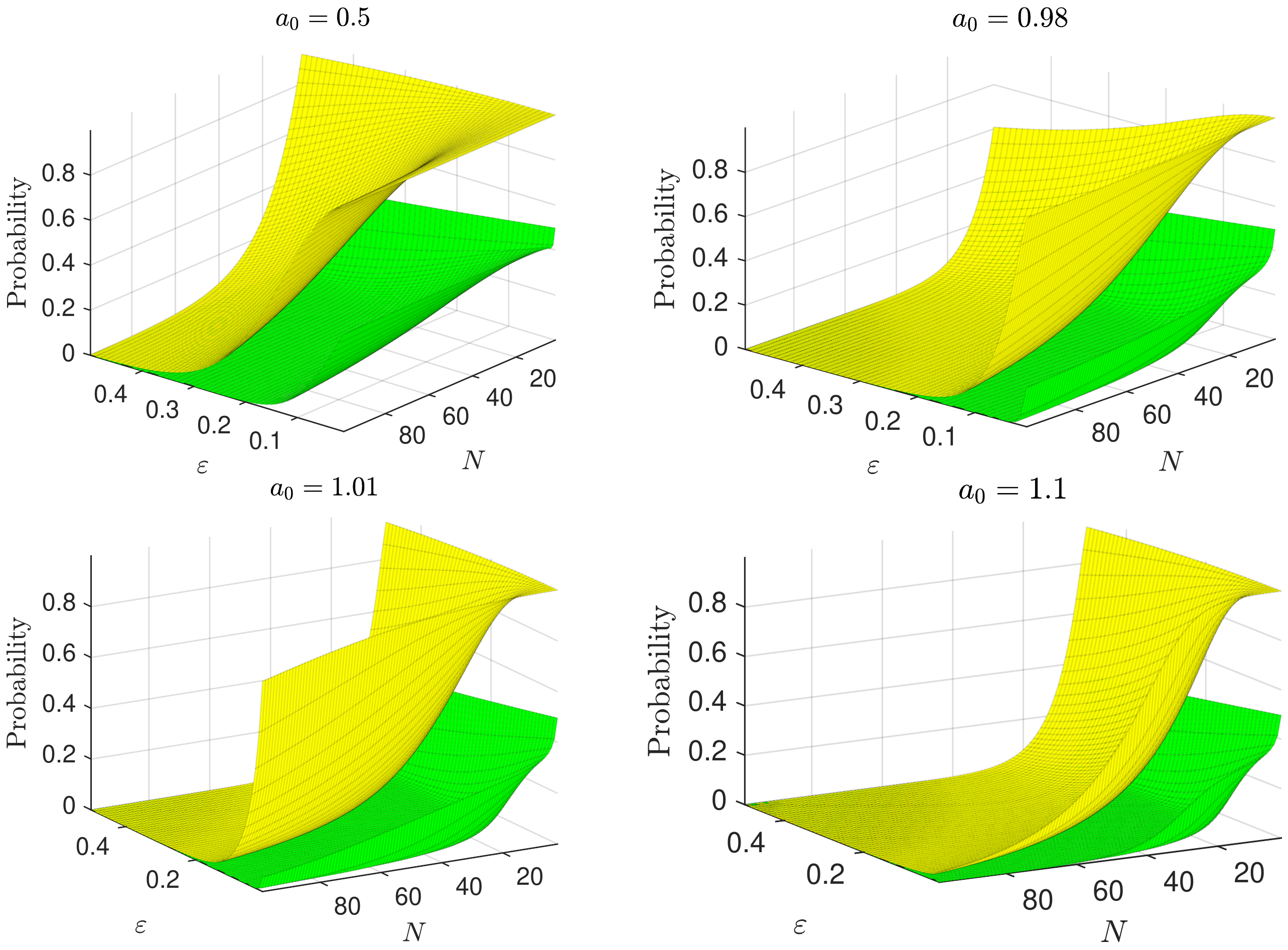}
		\caption{True deviation probability $\mathbb{P}(\hat{a}_N-a_0>\varepsilon)$ (green) and its upper bound (yellow) for $a_0 = 0.5, 0.98, 1.01$, and $1.1$.}
		\label{fig1}}
\end{figure} 

\begin{figure}[h!]
	\centering{
		\includegraphics[trim=0.0cm 0cm 0.cm 0cm,clip, width=0.48\textwidth]{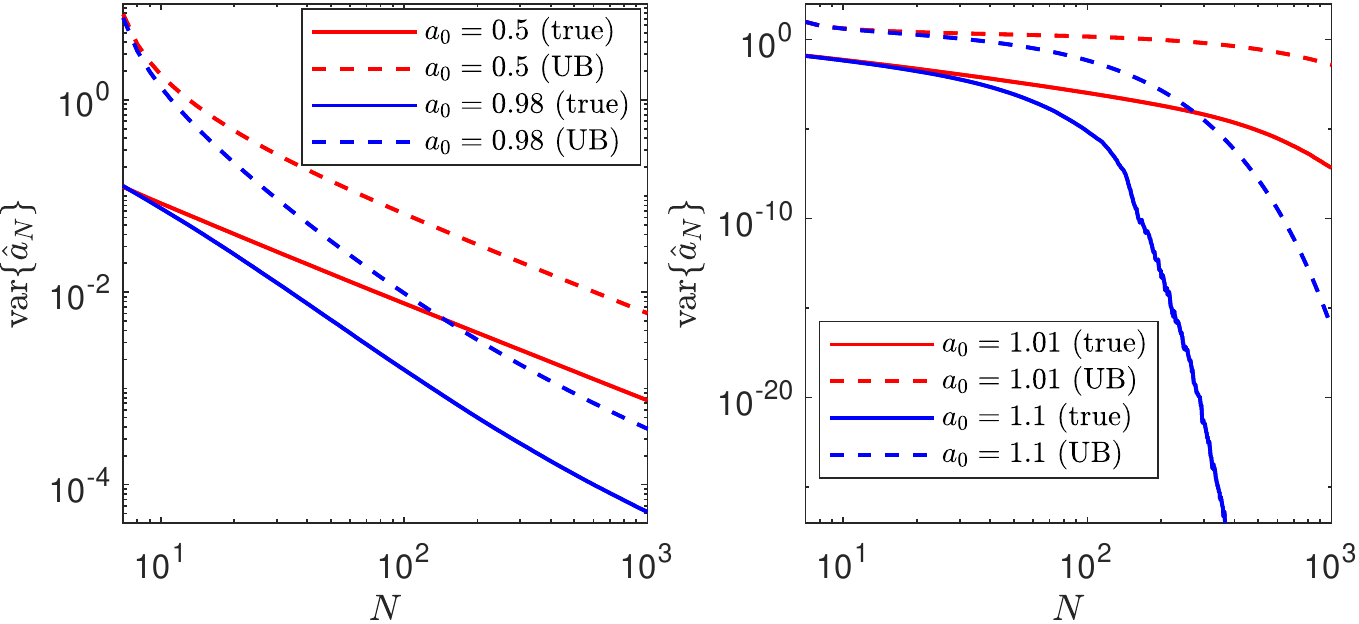}
		\caption{True (solid) and upper bounds (dashed) for var$\{\hat{a}_N\}$. Left: $a_0=0.5$ (red), $0.98$ (blue). Right: $a_0=1.01$ (red), $1.1$ (blue).\vspace{-0.4cm}}
		\label{fig2}}
\end{figure} 
For both stable and unstable cases, the bounds in deviation probability and variance are shown to have similar qualitative behavior to the simulated curves. Some conservatism was expected, due in part to Chernoff's inequality and by the square root introduced by the application of Proposition \ref{proposition1}. As the real parameter ranges from stable to unstable, the deviation probability flattens to zero quicker due to the increase of $\lambda_2$, which is a pivotal term in both bounds. For fixed $N$, this naturally leads to a reduction of variance for more unstable systems, which is also qualitatively predicted with our results.

\section{Conclusions}
\label{sec:conclusions}
In this paper we have obtained finite-sample upper bounds for the probability of deviation and variance of the least square estimator of the parameter of a scalar AR(1) process. Our results consider both stable and unstable cases, and reflect the correct qualitative behavior for both scenarios. For obtaining these bounds we have used elements of decoupling theory, which provides a promising approach to obtaining finite-sample bounds in dynamical systems.

Future work includes extending these results for general autoregressive processes with sub-Gaussian noise, and applying these bounds for performing a non-asymptotic analysis of general system identification model structures.

\newpage

\balance
\bibliographystyle{IEEEbib}
\bibliography{References}

\vfill\pagebreak
\newpage
\onecolumn
\section{Supplementary material}
\label{sec:appendix}
In this section we complete details on parts of the proofs of Theorems \ref{theorem1} and \ref{theorem2}.
\begin{fact}
The following identity holds:
\begin{equation}
\label{integralszego}
\frac{1}{2\pi} \int_{-\pi}^\pi \log\left(1+\frac{\varepsilon^2}{|e^{j\omega}-a_0|^2}\right) d\omega=\log\left(\frac{1+a_0^2 + \varepsilon^2}{2}+\sqrt{\left(\frac{1+a_0^2 + \varepsilon^2}{2}\right)^2-a_0^2} \right). 
\end{equation}
\end{fact}
\begin{proof}
By Szeg\"o's formula \cite{kailath2000linear}, the integral in question is equal to $\log(r_e)$, where $r_e$ is the constant factor of the spectral factorization of 
\begin{equation}
\Phi_{\tilde{y}}(z) = \frac{1+a_0^2 - a_0(z+z^{-1})+\varepsilon^2}{|z-a_0|^2}. \notag
\end{equation}
Here, $1+a_0^2-a_0(z+z^{-1})+\varepsilon^2 = z^{-1}(-a_0z^2+[1+a_0^2 + \varepsilon^2]z-a_0)$, and therefore the numerator of $\Phi_{\tilde{y}}(z)$ has roots at
\begin{equation}
z_1 = \frac{1+a_0^2+\varepsilon^2}{2a_0}-\frac{\sqrt{(1-a_0^2 + \varepsilon^2)^2+4a_0^2 \varepsilon^2}}{2a_0}, \quad z_2 = \frac{1+a_0^2+\varepsilon^2}{2a_0}+ \frac{\sqrt{(1-a_0^2 + \varepsilon^2)^2+4a_0^2 \varepsilon^2}}{2a_0}. \notag
\end{equation}
Thus,
\begin{equation}
\Phi_{\tilde{y}}(z) = \left(\sqrt{a_0z_2} \frac{z-z_1}{z-a_0}\right) \left(\sqrt{a_0z_2} \frac{z-z_1}{z-a_0}\right)^*. \notag
\end{equation}
With this factorization, we conclude that $r_e = a_0 z_2$, from which we obtain \eqref{integralszego}.
\end{proof}

\begin{fact}
The matrix $\bar{\mathbf{R}}_{N-1}/\sigma^2$ formed with entries
\begin{equation}
\bar{r}_{ij}=a_0^{|i-j|} \frac{a_0^{2\min\{i,j\}}-1}{a_0^2-1}, \quad i,j=1,\dots,N-1, \notag
\end{equation}
has a tridiagonal matrix inverse of the form 
\begin{equation}
\label{inverse}
\sigma^2\bar{\mathbf{R}}_{N-1}^{-1}= \begin{bmatrix}
a_0^2+1 & -a_0 & & &  0\\
-a_0 & a_0^2+1 & -a_0 & &  \\
& -a_0 & \ddots & \ddots & \\
&  & \ddots & a_0^2+1 & -a_0 \\
0&  &  & -a_0 & 1 
\end{bmatrix}.
\end{equation}
Furthermore, the Toeplitz matrix $\sigma^2\bar{\mathbf{R}}_{N-1}^{-1}+a_0^2 \bm{\eta}_{N-1}\bm{\eta}_{N-1}^\top$, where $\bm{\eta}_{N-1}$ is the $N-1$-th column of the identity matrix of size $N-1$, is non-singular for $|a_0|>1$.
\end{fact}

\begin{proof} Denote the matrix product between $\sigma^2\bar{\mathbf{R}}_{N-1}$ and the matrix in \eqref{inverse} as $\mathbf{J}$, and denote the elements of $\sigma^2\bar{\mathbf{R}}_{N-1}$ and the matrix in \eqref{inverse} as $\{r_{ij}\}$ and $\{\tilde{r}_{ij}\}$ respectively. We shall compute the elements of the matrix $\mathbf{J}$, labeled $\{\mathbf{J}_{ij}\}$.
	\begin{itemize}
		\item
		For $1<i<N-1$:
		\begin{equation}
		\mathbf{J}_{ii}= \sum_{k=1}^{N-1} r_{ik} \tilde{r}_{ki} =-a_0 r_{i,i-1}+(a_0^2+1)r_{ii}-a_0 r_{i,i+1} = \frac{-a_0^2(a_0^{2i-2}-1)+(a_0^2+1)(a_0^{2i}-1)-a_0^2(a_0^{2i}-1)}{a_0^2-1} =1. \notag   
		\end{equation}
		\item
		For $1<j<i\leq N-1$:
		\begin{equation}
		\mathbf{J}_{ij}= -a_0 r_{i,j-1}+(a_0^2+1)r_{ij}-a_0 r_{i,j+1} = \frac{-a_0^{i-j+2}(a_0^{2j-2}-1)+(a_0^2+1)a_0^{i-j}(a_0^{2j}-1)-a_0^{i-j}(a_0^{2j+2}-1)}{a_0^2-1}=0. \notag  
		\end{equation}
		\item
		For $1\leq i<j< N-1$:
		\begin{equation}
		\mathbf{J}_{ij}= -a_0 r_{i,j-1}+(a_0^2+1)r_{ij}-a_0 r_{i,j+1} = \frac{-a_0^{i-j}(a_0^{2i}-1)+(a_0^2+1)a_0^{j-i}(a_0^{2i}-1)-a_0^{j-i+2}(a_0^{2i}-1)}{a_0^2-1} =0. \notag   
		\end{equation}
		\item
		For $j=1, i>1$:
		\begin{equation}
		\mathbf{J}_{i1}= (a_0^2+1)r_{i1}-a_0 r_{i,2} = (a_0^2+1)a_0^{i-1}-a_0^{i-1}\left(\frac{a_0^{4}-1}{a_0^2-1}\right) =0. \notag   
		\end{equation}
		\item
		For $i<N-1, j=N-1$:
		\begin{equation}
		\mathbf{J}_{i,N-1}= -a_0 r_{i,N-2}+ r_{i,N-1} = -a_0 a_0^{N-i-2}\left(\frac{a_0^{2i}-1}{a_0^2-1}\right)+a_0^{N-i-1}\left(\frac{a_0^{2i}-1}{a_0^2-1}\right) =0. \notag   
		\end{equation}
	\end{itemize}
	Lastly, $\mathbf{J}_{11} = (a_0^2+1)- \frac{a_0^2(a_0^2-1)}{a_0^2-1}=1$ and $\mathbf{J}_{N-1,N-1} = -\frac{a_0^2(a_0^{2N-4}-1)}{a_0^2-1}+\frac{(a_0^{2N-2}-1)}{a_0^2-1}=1 $, from which we conclude that $\mathbf{J}=\mathbf{I}_{N-1}$.
	
	Regarding the invertibility of $\sigma^2\bar{\mathbf{R}}_{N-1}^{-1}+a_0^2 \bm{\eta}_{N-1}\bm{\eta}_{N-1}^\top$, we use Ger\v{s}gorin's circle Theorem\footnote{En fact, an explicit expression for the eigenvalues exists: it can be shown (see e.g. \cite{kulkarni1999eigenvalues}) that the eigenvalues of $\sigma^2\bar{\mathbf{R}}_{N-1}^{-1}+a_0^2 \bm{\eta}_{N-1}\bm{\eta}_{N-1}^\top$ are given by $\lambda_k = a_0^2+1-2|a_0|\cos\left( \frac{k\pi}{N} \right)$, $k=1,2,\dots,N-1$,	which implies that $\lambda_k \geq a_0^2+1-2|a_0| = (|a_0|-1)^2>0$ for $k=1,\dots,N-1$.} \cite{Horn2012} and the symmetry of this matrix to bound the spectrum of $\sigma^2\bar{\mathbf{R}}_{N-1}^{-1}+a_0^2 \bm{\eta}_{N-1}\bm{\eta}_{N-1}^\top$. This spectrum, which we call $\Lambda$, satisfies
	\begin{align}
	\Lambda &\subset \{\lambda \in \mathbb{R}: |\lambda-a_0^2 -1|\leq 2|a_0|\} \cup \{\lambda \in \mathbb{R}: |\lambda-a_0^2 -1|\leq |a_0|\} \notag \\
	&=\{\lambda \in \mathbb{R}: |\lambda-a_0^2 -1|\leq 2|a_0|\}, \notag 
	\end{align}
	which implies that each eigenvalue $\lambda_k$ of $\sigma^2\bar{\mathbf{R}}_{N-1}^{-1}+a_0^2 \bm{\eta}_{N-1}\bm{\eta}_{N-1}^\top$ must satisfy
	\begin{equation}
	-2|a_0|\leq \lambda_k-a_0^2 -1\leq 2|a_0| \quad \iff \quad 0<(|a_0|-1)^2\leq \lambda_k \leq (|a_0|+1)^2.  \notag
	\end{equation}
We note that similar results regarding inverses of covariance matrices have been studied before; we refer to \cite{shaman1969inverse} for more details.
\end{proof}

\begin{fact}
Denote $\bar{\mathbf{T}}_{N}$ as the tridiagonal Toeplitz matrix of size $N\times N$ with non-zero diagonal elements $\{-a_0,a_0^2+1,-a_0\}$. Then, the following determinant equalities hold:
\begin{align}
\textnormal{det} (\bar{\mathbf{T}}_{N-1})-a_0^2\det (\bar{\mathbf{T}}_{N-2}) &= 1, \notag \\
\det (\bar{\mathbf{T}}_{N-1}+\varepsilon^2\mathbf{I}_{N-1})-a_0^2\det (\bar{\mathbf{T}}_{N-2}+\varepsilon^2\mathbf{I}_{N-2}) &=\frac{(\lambda_2-1)\lambda_1^{N}-(\lambda_1-1)\lambda_2^{N}}{\lambda_2-\lambda_1}, \notag
\end{align}
where 
\begin{equation}
\lambda_1 = \frac{a_0^2+\varepsilon^2+1}{2}-\frac{\sqrt{(a_0^2+\varepsilon^2+1)^2-4a_0^2}}{2}, \quad \lambda_2 = \frac{a_0^2+\varepsilon^2+1}{2}+\frac{\sqrt{(a_0^2+\varepsilon^2+1)^2-4a_0^2}}{2}. \notag
\end{equation}
\end{fact}

\begin{proof}
	For $\textnormal{det}(\bar{\mathbf{T}}_{N-1})$, via Proposition \ref{proposition2} in Section \ref{auxiliary}, we know that the following determinant relation holds:
	\begin{equation}
	\det(\bar{\mathbf{T}}_{N-1}) = (a_0^2+1)\det(\bar{\mathbf{T}}_{N-2})-a_0^2\det(\bar{\mathbf{T}}_{N-3}), \notag
	\end{equation}
	with $\det(\bar{T}_{1})=\bar{T}_{1}=a_0^2+1$, and $\det(\bar{\mathbf{T}}_2) = a_0^4+a_0^2+1$. So, we see that 
	\begin{equation}
	\textnormal{det} (\bar{\mathbf{T}}_{N-1})-a_0^2\det (\bar{\mathbf{T}}_{N-2}) = \textnormal{det} (\bar{\mathbf{T}}_{N-2})-a_0^2\det (\bar{\mathbf{T}}_{N-3}) = \dots = \textnormal{det} (\bar{\mathbf{T}}_2)-a_0^2\det (\bar{T}_1)=1. \notag
	\end{equation}
	For the second equality, we denote $\tilde{\mathbf{T}}_{N-1}:= \bar{\mathbf{T}}_{N-1}+\varepsilon^2\mathbf{I}_{N-1}$. Due to the tridiagonal structure, by using Proposition \ref{proposition2} again, we know that the determinants of interest satisfy
	\begin{equation}
	\det(\tilde{\mathbf{T}}_{N-1}) = (a_0^2+\varepsilon^2+1)\det(\tilde{\mathbf{T}}_{N-2})-a_0^2\det(\tilde{\mathbf{T}}_{N-3}), \notag
	\end{equation}
	with $\det(\tilde{T}_{1})=a_0^2+\varepsilon^2+1$, and $\det(\tilde{\mathbf{T}}_2) = (a_0^2+\varepsilon^2+1)^2-a_0^2$. So, we write
	\begin{equation}
	\begin{bmatrix}
	\det(\tilde{\mathbf{T}}_{N-2}) \\ \det(\tilde{\mathbf{T}}_{N-1})
	\end{bmatrix} = 
	\begin{bmatrix}
	0 & 1 \\
	-a_0^2 & a_0^2+\varepsilon^2+1 
	\end{bmatrix}\begin{bmatrix}
	\det(\tilde{\mathbf{T}}_{N-3}) \\ \det(\tilde{\mathbf{T}}_{N-2})
	\end{bmatrix} = 
	\begin{bmatrix}
	0 & 1 \\
	-a_0^2 & a_0^2+\varepsilon^2+1 
	\end{bmatrix}^{N-3}\begin{bmatrix}
	\det(\tilde{T}_{1}) \\ \det(\tilde{\mathbf{T}}_{2})
	\end{bmatrix}. \notag 
	\end{equation}
	For simplicity, we denote $\alpha:=-a_0^2$, $\beta = a_0^2+\varepsilon^2+1$. We are interested in
	\begin{align}
	\det(\tilde{T}_{1})-a_0^2 \det(\tilde{\mathbf{T}}_{2})&=\begin{bmatrix}
	\alpha & 1
	\end{bmatrix}\begin{bmatrix}
	0 & 1 \\
	\alpha & \beta 
	\end{bmatrix}^{N-3}\begin{bmatrix}
	\beta \\ \beta^2+\alpha 
	\end{bmatrix} \notag \\
	&=\frac{-1}{\alpha\sqrt{\beta^2+4\alpha}} \begin{bmatrix}
	\alpha & 1
	\end{bmatrix}\begin{bmatrix}
	\lambda_2 & \lambda_1 \\
	-\alpha & -\alpha 
	\end{bmatrix}\begin{bmatrix}
	\lambda_1^{N-3} & 0 \\
	0 & \lambda_2^{N-3} 
	\end{bmatrix}
	\begin{bmatrix}
	-\alpha & -\lambda_1 \\
	\alpha & \lambda_2 
	\end{bmatrix}\begin{bmatrix}
	\beta \\ \beta^2+\alpha 
	\end{bmatrix} \notag \\
	&=\frac{-1}{\sqrt{\beta^2+4\alpha}} \begin{bmatrix}
	\lambda_2-1 & \lambda_1-1
	\end{bmatrix}\begin{bmatrix}
	\lambda_1^{N-3} & 0 \\
	0 & \lambda_2^{N-3} 
	\end{bmatrix}
	\begin{bmatrix}
	-\alpha\beta-\lambda_1 \beta^2-\lambda_1 \alpha \\ \alpha\beta+\lambda_2 \beta^2+\lambda_2 \alpha 
	\end{bmatrix}. \notag 
	\end{align}
	By noting that $\beta^2 \lambda_{1,2}+\alpha \beta = \beta \lambda_{1,2}^2$, we obtain
	\begin{equation}
	\det(\tilde{T}_{1})-a_0^2 \det(\tilde{\mathbf{T}}_{2})=\frac{-1}{\sqrt{\beta^2+4\alpha}} \begin{bmatrix}
	\lambda_2-1 & \lambda_1-1
	\end{bmatrix}\begin{bmatrix}
	\lambda_1^{N-3} & 0 \\
	0 & \lambda_2^{N-3} 
	\end{bmatrix}
	\begin{bmatrix}
	-\lambda_1^2\beta -\lambda_1 \alpha \\ \lambda_2^2\beta+\lambda_2 \alpha 
	\end{bmatrix} =\frac{(\lambda_2-1)\lambda_1^{N}-(\lambda_1-1)\lambda_2^{N}}{\lambda_2-\lambda_1}  \tag*{\qedhere}
	\end{equation}
\end{proof}

\begin{fact}[Choices for $x^*$ and $z^*$] We seek for a computable value for $x^*$ such that the bound \eqref{almostbound}
\begin{equation}
\textnormal{var}\{\hat{a}_N\} \leq 2x^* + \frac{8}{\sqrt[4]{z^*-a_0^2}} \left( \frac{|a_0|^{3-\frac{N}{2}}}{N-6}-\frac{|a_0|^{1-\frac{N}{2}}}{N+2} \right), \notag
\end{equation}
is as tight as possible. Here, $z^*$ is computed as
\begin{equation}
z^* = \frac{1+a_0^2+x^*}{2}+ \frac{1}{2}\sqrt{(1-a_0^2+x^*)^2-4a_0^2}. \notag
\end{equation}
where $x^*>0$. An appropriate value for $x^*$ should be directly related to the change of regime in the probability bound \eqref{bound1}. Valid choices for $x^*$ and $z^*$ are
\begin{equation}
x^*=\frac{|a_0|^{4-\frac{N}{2m}}(N+4m)}{N}, \quad z^*=a_0^2+\frac{|a_0|^{4-\frac{N}{2m}}(N+4m)}{N}. \notag
\end{equation}
\end{fact}

\begin{proof}
We are interested in the value of $\varepsilon$ when the following equality holds
	\begin{equation}
	\frac{\lambda_2-\lambda_1}{(1-\lambda_1)\lambda_2^{\frac{N}{4m}}} = 1. \notag
	\end{equation}
	If we denote $\lambda_2=z$, we have $\lambda_1 = a_0^2/z$, and hence the equality can be written as
	\begin{equation}
	z^{\frac{N}{4m}+1}-a_0^2z^{\frac{N}{4m}}-z^2+a_0^2=0. \notag
	\end{equation}
	Since $\lambda_1<1$, we require that $z>a_0^2$. Hence, we propose $z=a_0^2+\xi$, where $\xi>0$. The goal is to find an appropriate $\xi$ now. If we write $f(z):=z^{\frac{N}{4m}+1}-a_0^2z^{\frac{N}{4m}}-z^2+a_0^2$, we see that $f(a_0^2)<0$. So, an appropriate $\xi$ would satisfy $f(a_0^2+\xi)>0$. For this, it is sufficient that
	\begin{align}
	f(a_0^2+\xi) &= \xi(a_0^2+\xi)^{\frac{N}{4m}}-(a_0^4+2a_0^2 \xi + \xi^2 - a_0^2)  \notag \\
	&> \xi\left(|a_0|^{\frac{N}{2m}}+\frac{N}{4m}|a_0|^{2(\frac{N}{4m}-1)}\xi\right)-\xi^2-2a_0^2 \xi+a_0^2-a_0^4 \notag \\
	&=0. \notag
	\end{align}
	This is simply a quadratic equation in $\xi$. So, we can solve for $\xi$:
	\begin{align}
	\xi &= \frac{2a_0^2-|a_0|^{\frac{N}{2m}}+\sqrt{|a_0|^{\frac{N}{m}}-4|a_0|^{\frac{N}{2m}+2}-\frac{N}{m} |a_0|^{\frac{N}{2m}}+\frac{N}{m}|a_0|^{\frac{N}{2m}+2}+4a_0^2}}{2\left( \frac{N}{4m}|a_0|^{\frac{N}{2m}-2}-1 \right)}\notag \\
	&<\frac{2a_0^2-|a_0|^{\frac{N}{2m}}+\sqrt{|a_0|^{\frac{N}{m}}+\frac{N}{m}|a_0|^{\frac{N}{2m}+2}}}{\frac{N}{2m}|a_0|^{\frac{N}{2m}-2}}\notag \\
	&<\frac{4m|a_0|^{4-\frac{N}{2m}}}{N}-\frac{2m a_0^2}{N}+ \frac{2m a_0^2}{N} +|a_0|^{4-\frac{N}{2m}} \notag\\
	&=\frac{|a_0|^{4-\frac{N}{2m}}(N+4m)}{N}.\notag
	\end{align}
	So, thanks to the substitutions we have made in the integrals that lead to \eqref{almostbound}, we can see that a choice for $z^*$ can actually be 
	\begin{equation}
	z^*=a_0^2+\frac{|a_0|^{4-\frac{N}{2m}}(N+4m)}{N}, \quad m \geq \frac{1}{4}. \notag
	\end{equation}
	Finally, we proceed to obtain a bound on $x^*$. Here, we denote $\xi^*:=z^*-a_0^2$. Note that the relationship between $z^*$ and $x^*$ is given by
	\begin{equation}
	z^* = \frac{1+a_0^2+x^*}{2} + \sqrt{\left( \frac{1+a_0^2+x^*}{2} \right)^2-a_0^2} \iff \frac{{z^*}^2+a_0^2}{2z^*} = \frac{1+a_0^2+x^*}{2}, \notag
	\end{equation}
	from which we conclude that
	\begin{equation}
	x^*=z^*+\frac{a_0^2}{z^*}-1-a_0^2 =\xi^*+\frac{a_0^2}{a_0^2+\xi^*}-1 =\xi^*\left(1-\frac{1}{a_0^2+\xi^*}\right)>0. \notag
	\end{equation}
	For simplicity, we can bound $x^*$ by $\xi^*$.
\end{proof}

\subsection{Auxiliary results}
\label{auxiliary}
\begin{proposition}[Determinant of tridiagonal matrices]
	\label{proposition2}
Consider the sequence of Toeplitz tridiagonal matrices $\{\mathbf{T}_N\}$ of increasing size $N$ and non-diagonal elements $\{\alpha,\beta,\gamma\}$. Then,
\begin{equation}
\label{tridiagproof}
\det(\mathbf{T}_N) = \beta \det(\mathbf{T}_{N-1})-\alpha \gamma \det(\mathbf{T}_{N-2}).
\end{equation}	
\end{proposition}
\begin{proof}
	By Laplace's expansion formula, we find that $\det(\mathbf{T}_N) = \beta \det(\mathbf{T}_{N-1})-\gamma \det(\mathbf{S}_{N-1})$, where
	\begin{equation}
	\mathbf{S}_{N-1} = \begin{bmatrix}
	\alpha & \gamma & 0 & &  0\\
	0 & \beta & \gamma & &  \\
	& \alpha & \ddots & \ddots & \\
	&  & \ddots & \beta & \gamma \\
	0&  &  & \alpha & \beta 
	\end{bmatrix} \in \mathbb{R}^{(N-1)\times(N-1)}. \notag
	\end{equation}
	Clearly, $\det(\mathbf{S}_{N-1}) = \alpha \det(\mathbf{T}_{N-2})$, and hence \eqref{tridiagproof} follows.
\end{proof}

\begin{proposition}[Quotients of determinants of symmetric Toeplitz matrices]
	\label{proposition3}
Consider the sequence $\{\mathbf{T}_n\}$ of symmetric Toeplitz matrices, each one with elements $\{t_i\}_{i=0}^{n-1}$. Then, for $N>1$,
\begin{equation}
\label{quotientsdet}
\frac{\det(\mathbf{T}_{N+1})}{\det(\mathbf{T}_N)} \leq \frac{\det(\mathbf{T}_{N})}{\det(\mathbf{T}_{N-1})}.
\end{equation} 
\end{proposition}

\begin{proof}
	Consider the weakly stationary process $\{X_i\}$, where $X_i$ is a zero-mean Gaussian random variable. Furthermore, assume that the process exhibits covariance $\mathbb{E}[X_k X_j]=t_{k-j}$. Let us solve the problem of finding the best linear predictor of $X_N$ based on $N$ previous values of the random process $\{X_i\}$:
	\begin{equation}
	\hat{X}_N = \sum_{i=1}^N a_i X_{N-i}. \notag
	\end{equation}
	If the cost is $\mathbb{E}[(X_N-\hat{X}_N)^2]$, then it is well known (see, e.g. \cite{kay1993fundamentals}) that the optimal predictor is in fact
	\begin{equation}
	\hat{X}_N^{\textnormal{opt}} = \mathbb{E}[X_N|X_0,\dots,X_{N-1}] = \mathbf{K}_{X_N \mathbf{X}^{N-1}}\mathbf{K}_{\mathbf{X}^{N-1}}^{-1} \mathbf{X}^{N-1}, \notag
	\end{equation}
	where
	\begin{equation}
	\mathbf{K}_{X_N \mathbf{X}^{N-1}} = \begin{bmatrix}
	t_1 & t_2 & \dots & t_{N}
	\end{bmatrix}, \quad \mathbf{K}_{\mathbf{X}^{N-1}}= \begin{bmatrix}
	t_0 & t_1 & \cdots & t_{N-1} \\
	t_1 & t_0 & \cdots & t_{N-2} \\
	\vdots & \vdots & \ddots & \vdots \\
	t_{N-1} & t_{N-2} & \cdots & t_0
	\end{bmatrix}, \quad \mathbf{X}^{N-1}= \begin{bmatrix}
	X_{N-1}\\ X_{N-2} \\ \vdots \\ X_0
	\end{bmatrix}. \notag
	\end{equation}
	Note that $\mathbf{K}_{\mathbf{X}^{N-1}} = \mathbf{T}_{N-1}$. Hence, the optimal cost is given by
	\begin{equation}
	\label{increasingcomplex}
	\mathbb{E}[(X_N-\hat{X}_N^{\textnormal{opt}})^2] = t_0 - \mathbf{K}_{X_N \mathbf{X}^{N-1}}\mathbf{K}_{\mathbf{X}^{N-1}}^{-1}\mathbf{K}_{\mathbf{X}^{N-1} X_N} = \frac{\det(\mathbf{K}_{\mathbf{X}^N})}{\det(\mathbf{K}_{\mathbf{X}^{N-1}})} = \frac{\det(\mathbf{T}_{N})}{\det(\mathbf{T}_{N-1})}. 
	\end{equation}
	Now, if we find the best linear predictor based on $N+1$ previous values, following the same computations we reach the optimal cost $\det(\mathbf{T}_{N+1})/\det(\mathbf{T}_{N})$. This quantity cannot be greater than \eqref{increasingcomplex}, due to the increasing complexity of the predictor. Thus, \eqref{quotientsdet} follows.
\end{proof}

\end{document}